\documentclass[11pt]{article}

\usepackage{graphicx}%
\usepackage{amsmath,amssymb,amsthm,upref,bm}%
\usepackage{refcheck}

\newtheorem{corollary}{Corollary}%
\newtheorem{theorem}{Theorem}%
\newtheorem{proposition}{Proposition}%
\begin{document}

\baselineskip=4.4mm

\makeatletter

\newcommand{\E}{\mathrm{e}\kern0.2pt} 
\newcommand{\D}{\mathrm{d}\kern0.2pt}
\newcommand{\RR}{\mathbb{R}}
\newcommand{\CC}{\mathbb{C}}%
\newcommand{\ii}{\kern0.05em\mathrm{i}\kern0.05em}

\renewcommand{\Re}{\mathrm{Re}} 
\renewcommand{\Im}{\mathrm{Im}}

\def\bottomfraction{0.9}

\title{\bf Weighted means of harmonic functions \\ and characterization of balls}

\author{Nikolay Kuznetsov}

\date{}

\maketitle

\vspace{-6mm}

\begin{center}
Laboratory for Mathematical Modelling of Wave Phenomena, \\ Institute for Problems
in Mechanical Engineering, Russian Academy of Sciences, \\ V.O., Bol'shoy pr. 61, St
Petersburg 199178, Russian Federation \\ E-mail: nikolay.g.kuznetsov@gmail.com
\end{center}

\begin{abstract}
\noindent Weighted mean value identities over balls are considered for harmonic
functions and their derivatives. Logarithmic and other weights are involved in these
identities for functions. Some applications of weighted identities are presented.
Also, new analytic characterizations of balls are proved; each of them requires the
volume mean of a single weight function over the domain under consideration to be
equal to a prescribed number depending on the weight.
\end{abstract}

\setcounter{equation}{0}


\section{Weighted means of harmonic functions}

{\bf 1.1. Background.} A function $u \in C^2 (D)$ is called harmonic (the origin of
this term is described in \cite{ABR}, p.~25), if it satisfies the equation $\nabla^2
u (x) = 0$ in a domain $D \subset \RR^m$, $m \geq 2$; $\nabla = (\partial_1, \dots ,
\partial_m)$ denotes the gradient operator, $\partial_i = \partial / \partial x_i$
and $x = (x_1, \dots, x_m)$ is a point of $\RR^m$. Studies of mean value properties
of harmonic functions date back to the Gauss theorem of the arithmetic mean over a
sphere; see \cite{G}, Article~20. Nowadays, its standard formulation is as follows.

\begin{theorem}
Let $u \in C^2 (D)$ be harmonic in a domain $D \subset \RR^m$, $m \geq 2$. Then for
every $x \in D$ the equality $M^\circ (x, r, u) = u (x)$ holds for each admissible
sphere $S_r (x)$.
\end{theorem}

Here and below the following notation and terminology are used. The open ball of
radius $r$ centred at $x$ is $B_r (x) = \{ y : |y-x| < r \}$; it is called
admissible with respect to a domain $D$ provided $\overline{B_r (x)} \subset D$, and
$S_r (x) = \partial B_r (x)$ is the corresponding admissible sphere. If $u \in C^0
(D)$, then its spherical mean value over $S_r (x) \subset D$ is
\begin{equation}
M^\circ (x, r, u) = \frac{1}{|S_r|} \int_{S_r (x)} u (y) \, \D S_y =
\frac{1}{\omega_m} \int_{S_1 (0)} u (x +r y) \, \D S_y \, ; \label{sm}
\end{equation}
here $|S_r| = \omega_m r^{m-1}$ and $\omega_m = 2 \, \pi^{m/2} / \Gamma (m/2)$ is
the total area of the unit sphere (as usual $\Gamma$ stands for the Gamma function),
and $\D S$ is the surface area measure.

Integrating $M^\circ (x, r, u)$ with respect to $r$, one obtains the following mean
value property over balls.

\begin{theorem}
Let $u \in C^2 (D)$ be harmonic in a domain $D \subset \RR^m$, $m \geq 2$. Then for
every $x \in D$ the equality $M^\bullet (x, r, u) = u (x)$ holds for each admissible
ball $B_r (x)$.
\end{theorem}

Here, the volume mean of a locally Lebesgue integrable function $u$ is defined
similarly to \eqref{sm}:
\begin{equation*}
M^\bullet (x, r, u) = \frac{1}{|B_r|} \int_{B_r (x)} u (y) \, \D y =
\frac{m}{\omega_m r^m} \int_{|y| < r} u (x + y) \, \D y \, . \label{bm}
\end{equation*}

In their extensive article \cite{NV}, Netuka and Vesel\'y reviewed many other
assertions involving various mean value properties of harmonic functions. The survey
\cite{Ku2} published several years ago substantially complemented \cite{NV} with old
and new results not covered in \cite{NV}. However, to the best author's knowledge,
no results concerning weighted means of harmonic functions have appeared so far. The
aim of the present note is to fill in this gap at least partially.

\vskip7pt {\bf 1.2. Weighted means.} In the recent preprint \cite{Ku4}, the author
derived heuristically the two-dimensional version of the following.

\begin{theorem}
Let $u \in C^2 (D)$ be harmonic in a domain $D \subset \RR^m$, $m \geq 2$. Then
\begin{equation}
u (x) = \frac{m}{|B_r|} \int_{B_r (x)} u (y) \log \frac{r}{|x-y|} \, \D y 
\label{har}
\end{equation}
for every $x \in D$ and each admissible ball $B_r (x)$.
\end{theorem}

\begin{proof}
Theorem 2 implies that the right-hand side of \eqref{har} is equal to
\begin{equation}
u (x) \, m \log r - \frac{m}{|B_r|} \int_{B_r (x)} u (y) \log |x-y| \, \D y \, .
\label{har'}
\end{equation}
In the polar coordinates $(\rho, \theta)$ centred at $x$, we have
\begin{eqnarray*}
&& \!\!\!\!\!\!\!\!\!\! \int_{B_r (x)} \!\! u (y) \log |x-y| \, \D y = \! \int_0^r
\! \int_{S_1 (0)} \!\! u (\rho , \theta) \, \rho^{m-1} \log \rho \, \D S \, \D \rho
\\ && \ \ \ \ \ \ \ \ \ \ \ \ \ \ \ \ \ \ \ \ \ \ \ \ \ \ = |S_1| \, u (x)
\int_0^r \!\! \rho^{m-1} \log \rho \, \D \rho \, ,
\end{eqnarray*}
where the last equality is a consequence of Theorem 1. Since
\[ \int_0^r \!\! \rho^{m-1} \log \rho \, \D \rho = r^m ( m \log r - 1 ) / m^2 \, , 
\]
the expression \eqref{har'} is equal to $u (x)$, which completes the proof.
\end{proof}

The two-dimensional version of Theorem 3 formulated in \cite{Ku4} was proved by
Osch\-mann \cite{O}.

\begin{corollary}
Let $D$ be a domain in $\RR^m$. If $u$ is harmonic in $D$, then
\begin{equation*}
\partial_i u (x) = \frac{m}{|B_r|} \int_{B_r (x)} u (y) \, \frac{y_i -
x_i}{|x-y|^2} \, \D y \, , \quad i = 1,2 ,
\end{equation*}
for every $x \in D$ and each admissible ball $B_r (x)$.
\end{corollary}

\begin{proof}
Since $\partial_i u$ is also harmonic in $D$, Theorem 3 implies that
\begin{equation*}
\partial_i u (x) = \frac{m}{|B_r|} \int_{B_r (x)} \frac{\partial u}{\partial y_i}
\log \frac{r}{|x-y|} \, \D y \, , \quad i = 1,2 
\end{equation*}
for every $x \in D$, where $B_r (x)$ is an arbitrary admissible ball. Integrating by
parts on the right-hand side, one arrives at the required assertion because the
integral over $S_r (x)$ vanishes. Indeed, $\log \frac{r}{|x-y|} = 0$ when $y \in S_r
(x)$.
\end{proof}

Corollary 1 makes obvious the following.

\begin{proposition}
Let $u$ be harmonic in $D$. If $D'$ is a compact subset of $D$, then
\[ \max_{x \in D'} |\partial_i u| \leq \frac{m}{d} \sup_{x \in D} |u| \quad for 
\ i=1,\dots,m ,
\]
where $d$ is the distance between $D'$ and $\partial D$.

Moreover, if $u \geq 0$ in $D$, then
\[ |\partial_i u (x_0)| \leq \frac{m}{d_0} u (x_0) \quad for \ i=1,\dots,m ,
\]
where $d_0$ is the distance from $x_0$ to $\partial D$.
\end{proposition}

Along with $\log \frac{r}{|x-y|}$, there are other weights with mean value
properties analogous to \eqref{har}; a couple of them is adduced in the following.

\begin{theorem}
Let $u \in C^2 (D)$ be harmonic in a domain $D \subset \RR^m$, $m \geq 2$. Then
\begin{equation}
(m \, \alpha^{-1} - 1) \, r^{\alpha - m} \, u (x) = \frac{1}{|B_r|} \int_{B_r (x)} u
(y) \left[ |x-y|^{\alpha - m} - r^{\alpha - m} \right] \D y \, , \ \ \alpha \in (0,
m) \, , \label{11a}
\end{equation}
and
\begin{equation}
\left[ 1 - \frac{m}{m + \beta} \right] r^\beta \, u (x) = \frac{1}{|B_r|} \int_{B_r
(x)} u (y) \left[ r^\beta - |x - y|^\beta \right] \D y \, , \quad \beta > 0 \, ,
\label{11b}
\end{equation}
for every $x \in D$ and each admissible ball $B_r (x)$. Here both weights are
integrable functions of $y$ over any bounded domain for any $x \in \RR^m$.
\end{theorem}

The proof literally repeats that of Theorem 3. Notice that the coefficients at $u$
are positive on the left-hand side of \eqref{11a} and \eqref{11b} in view of the
assumptions about $\alpha$ and $\beta$, respectively.

\section{Characterization of balls via averaging weights}

In 1962, Epstein published the one-page long note \cite{E}, in which he proved the
following theorem.

\begin{quote}
{\it Let $D$ be a simply connected plane domain of finite area and $t$ a point of
$D$ such that, for every function $u$ harmonic in $D$ and integrable over $D$, the
mean value of $u$ over the area of $D$ equals $u (t)$. Then $D$ is a disc and $t$ is
its center.}
\end{quote}

\noindent Further studies of inverse mean value properties (this widely accepted
term, coined by Hansen and Netuka \cite{HN1}, concerns analytic characterizations
 of various domains via their volume or/and boundary area means) for almost six decades
were restricted to those of {\it harmonic functions}; see Sections~7 and 8 of the
extensive survey~\cite{NV}. Recently, inverse mean value properties were obtained
for real-valued solutions of the modified Helmholtz and Helmholtz equations; see the
notes \cite{Ku1} and \cite{Ku3}, respectively.

However, it occurs that one can characterize balls without involving any solutions
of partial differential equations. Establishing some characterizations of this kind,
which are based on weighted means of domains' volume, is the second aim of the
present note.

The motivation behind this is the observation that the weighted mean value formula
\eqref{har} takes the form
\begin{equation}
m^{-1} = \frac{1}{|B_r|} \int_{B_r (x)} \log \frac{r}{|x-y|} \, \D y
\label{MW'}
\end{equation}
for $u$ equal to a nonzero constant. Another essential point is that the
spherically symmetric weight $\log \, (r / |x - y|)$ has the following properties:
it is a continuous function of $y$ going along any ray emanating from $x$; moreover,
this function decreases monotonically from $+ \infty$ to $- \infty$ and it is
positive when $y$ belongs to $B_r (x)$, being negative outside this ball. This
particular behaviour of the weight is used in the following.

\begin{theorem}
Let $D \subset \RR^m$ be a bounded domain. If the identity
\begin{equation}
\frac{1}{m} = \frac{1}{|D|} \int_D \log \frac{r}{|x - y|} \, \D y \label{MW''}
\end{equation}
holds with $x \in D$ and $r > 0$ such that $|D| \geq |B_r|$, then $D = B_r (x)$.
\end{theorem}

\begin{proof}
Without loss of generality, we suppose that the domain $D$ is located so that $x$
coincides with the origin. Let us show that the assumption $D \neq B_r (0)$ leads to
a contradiction. For this purpose we consider two bounded open sets $G_i = D
\setminus \overline{B_r (0)}$ (nonempty by the assumption about $D$ and $r$) and
$G_e = B_r (0) \setminus \overline D$ (possibly empty).

Let us write \eqref{MW''} as follows:
\begin{equation}
\frac{|D|}{m} = \int_D \log \frac{r}{|y|} \, \D y \, , \label{1}
\end{equation}
Since identity \eqref{MW'} holds for $x = 0$ and $B_r (0)$, we write it in the same
way:
\begin{equation}
\frac{|B_r|}{m} = \int_{B_r (0)} \log \frac{r}{|y|} \, \D y \, . \label{2}
\end{equation}
Subtracting \eqref{2} from \eqref{1}, we obtain
\begin{equation*}
\frac{|D| - |B_r|}{m} = \int_{G_i} \log \frac{r}{|y|} \, \D y - \int_{G_e} \log
\frac{r}{|y|} \, \D y \, .
\end{equation*}
Here the difference on the right-hand side is negative. Indeed, $\log \, (r / |y|) <
0$ on $G_i \neq \emptyset$, because $|y| > r$ there. Hence, the first term is
negative. If $G_e \neq \emptyset$, then the second integral is positive because
$\log \, (r / |y|) > 0$ on $G_e$, where $|y| < r$. On the other hand, the expression
on the left-hand side is nonnegative. The obtained contradiction proves the theorem.
\end{proof}

An analogue of Theorem 5 ensues by averaging the weight
\begin{equation}
|x-y|^{\alpha - m} - r^{\alpha - m} , \quad \mbox{where} \ r > 0 \ \mbox{and} \
\alpha \in (0, m) , \label{Riesz}
\end{equation}
used in the mean value formula \eqref{11a}, which takes the form
\begin{equation}
\frac{1}{|B_r|} \int_{B_r (x)} \left[ |x-y|^{\alpha - m} - r^{\alpha - m} \right] \D
y = (m \, \alpha^{-1} - 1) \, r^{\alpha - m} > 0 \ \ \mbox{for any} \ x \in \RR^m 
\label{26a}
\end{equation}
when $u$ is a nonzero constant. This identity is similar to \eqref{MW'} and allows
us to prove the following.

\begin{theorem}
Let $D \subset \RR^m$ be a bounded domain. If the identity
\begin{equation}
(m \, \alpha^{-1} - 1) \, r^{\alpha - m} = \frac{1}{|D|} \int_D \left[
|x-y|^{\alpha - m} - r^{\alpha - m} \right] \D y \label{26b}
\end{equation}
holds with $x \in D$ and $r > 0$ such that $|D| \geq |B_r|$, then $D = B_r (x)$.
\end{theorem}

The proof literally follows that of Theorem 4, but identities \eqref{26a} and
\eqref{26b} must be used instead of \eqref{MW'} and~\eqref{MW''}, respectively.
Therefore, assuming that $D \neq B_r (x)$ for $x \in D$, one arrives at the equality
\[ (|D| - |B_r|) \, (m \, \alpha^{-1} - 1) \, r^{\alpha - m} = \left[ \int_{G_i} -
\int_{G_e} \right] \left[ |x-y|^{\alpha - m} - r^{\alpha - m} \right] \D y \, ,
\]
which is impossible. Indeed, the expression on the left-hand side is nonnegative,
whereas the integral over $G_i \neq \emptyset$ is negative because $|x-y|^{\alpha -
m} < r^{\alpha - m}$ on $G_i$ and the integral over $G_e$ (possibly empty) is
positive, since $|x-y|^{\alpha - m} > r^{\alpha - m}$ on $G_e$ when it is nonempty.

It is worth mentioning the result obtained by O'Hara \cite{H}, which seems to be
related to Theorem~6. It is known that the first term in \eqref{Riesz} integrated
over $D \times D$ defines the so-called Riesz $(\alpha - m)$-energy of a domain $D$.
Its generalization proposed in \cite{H} serves for characterization of balls; see
Theorem~3.4 of that paper. However, its proof is rather technical.

Finally, the nonsingular weight function
\[ r^\beta - |x - y|^\beta , \ \ \mbox{where} \ \beta > 0 \, ,
\]
used in the mean value formula \eqref{11b}, which takes the form
\[ \frac{1}{|B_r|} \int_{B_r (x)} \left[ r^\beta - |x - y|^\beta \right] \D y = 
\left[ 1 - \frac{m}{m + \beta} \right] r^\beta > 0 \ \ \mbox{when} \ r > 0 
\]
and $u$ equals to a nonzero constant, allows us to prove the assertion analogous to
Theorems~5 and~6, but using the latter weight. Thus, one arrives at another
characterization of balls via averaging a weight over the domain's volume.

It is easy to continue the list of weights (singular and nonsingular) that
characterize balls via averaging. Indeed, one just has to take into account that
such a weight is any continuous (of course, this requirement can be relaxed)
function
\[ (0, \infty) \times (0, \infty) \ni (t, r) \mapsto w \in \RR
\]
with the following properties:

(i) $w (t, r) > 0$ when $t < r$, $w (t, r) < 0$ when $t > r$ and $w (r, r) = 0$;

(ii) for every $x \in \RR^m$ and every $r > 0$, the superposition $w (|x-y|, r)$ is
a locally integrable function of $y \in \RR^m$.

It is interesting to find out whether every weight characterizing balls via
averaging yields a mean value identity for harmonic functions. Also, the converse of
this assertion is worth considering.

\end{document}